\documentclass{amsart}

\newtheorem{theorem}{Theorem}
\newtheorem{lemma}[theorem]{Lemma}

\newtheorem{corollary}[theorem]{Corollary}

\newtheorem{example}[theorem]{Example}

\theoremstyle{definition}

\newtheorem{remark}[theorem]{Remark}

\usepackage{amscd,amssymb}

\begin{document}

\title[Inner Automorphisms]
{Inner Automorphisms of Lie Algebras\\
Related with Generic $2\times 2$ Matrices}

\author[Vesselin Drensky and \c{S}ehmus F\i nd\i k]
{Vesselin Drensky and \c{S}ehmus F\i nd\i k}
\address{Institute of Mathematics and Informatics,
Bulgarian Academy of Sciences,
1113 Sofia, Bulgaria}
\email{drensky@math.bas.bg}
\address{Department of Mathematics,
\c{C}ukurova University, 01330 Balcal\i,
 Adana, Turkey}
\email{sfindik@cu.edu.tr}

\thanks
{The research of the first named author was partially supported
by Grant for Bilateral Scientific Cooperation between Bulgaria and Ukraine}
\thanks
{The research of the second named author was partially supported by the
 Council of Higher Education (Y\"OK) in Turkey}


\begin{abstract}
Let $F_m=F_m(\text{\rm var}(sl_2(K)))$ be the relatively free algebra of rank $m$
in the variety of Lie algebras generated by the algebra $sl_2(K)$ over a field $K$
of characteristic 0. Translating an old result of Baker from 1901 we present a
multiplication rule for the inner automorphisms of the completion $\widehat{F_m}$
of $F_m$ with respect to the formal power series topology.
Our results are more precise for $m=2$ when $F_2$ is isomorphic to
the Lie algebra $L$ generated by two generic
traceless $2\times 2$ matrices.
We give a complete description of the group of inner
automorphisms of $\widehat L$.
As a consequence we obtain similar results for the
automorphisms of the relatively free algebra
$F_m/F_m^{c+1}=F_m(\text{\rm var}(sl_2(K))\cap {\mathfrak N}_c)$
in the subvariety of $\text{\rm var}(sl_2(K))$ consisting of all
nilpotent algebras of class at most $c$
in $\text{\rm var}(sl_2(K))$.
\end{abstract}

\maketitle

\section*{Introduction}
Let $L_m$ be the free Lie algebra of rank $m\geq 2$ over a
field $K$ of characteristic 0 and let $G$ be an arbitrary Lie algebra.
Let $I(G)=I_m(G)$ be the ideal of $L_m$ consisting of all Lie polynomial
identities in $m$ variables for the algebra $G$. The factor algebra
$F_m(G)=F_m(\text{\rm var }G)=L_m/I(G)$ is the relatively free Lie algebra of rank $m$
in the variety of Lie algebras generated by $G$.
Typical examples of relatively free algebras are free solvable of class $k$ Lie algebras
when $I(G)=L_m^{(k)}$
(e.g., free metabelian Lie algebras with $I(G)=L_m''$),
free nilpotent of class $c$ Lie algebras when $I(G)=L^{c+1}$,
relatively free algebras in a variety generated by a finite dimensional
simple Lie algebra $G$, etc. See the books by Bahturin \cite{Ba} and
Mikhalev, Shpilrain and Yu \cite{MSY} for a background on relatively free Lie algebras
and their automorphisms, respectively.

Let $F_m(G)$ be a relatively free Lie algebra freely generated by $x_1,\ldots,x_m$.
An automorphism $\varphi$ of $F_m(G)$ is called linear if it is of the form
\[
\varphi(x_j)=\sum_{i=1}^m\alpha_{ij}x_i,\quad \alpha_{ij}\in K,\quad j=1,\ldots,m.
\]
It is triangular if
\[
\varphi(x_j)=\alpha_jx_j+f_j(x_{j+1},\ldots,x_m),\quad 0\not=\alpha_j\in K,f_j\in F_m(G),\quad j=1,\ldots,m.
\]
(Here the polynomials $f_j=f_j(x_{j+1},\ldots,x_m)$ do not depend on the variables $x_1,\ldots,x_j$.)
The automorphism $\varphi$ is tame if it can be presented as a product of linear
and triangular automorphisms. Otherwise $\varphi$ is wild.

Cohn \cite{C} showed that every automorphism of the
free Lie algebra $L_m$ is tame. In particular, the group of automorphisms $\text{\rm Aut}(L_2)$ is isomorphic to
the general linear group $GL_2(K)$.
Quite often relatively free algebras $F_m(G)$
possess wild automorphisms and for better understanding of the group
$\text{\rm Aut}(F_m(G))$ one studies its important subgroups.

When we consider a finite
dimensional simple Lie algebra $G$ over $\mathbb C$,
the general theory gives that the series
\[
\exp(\text{\rm ad}g)=\sum_{n\geq 0}\frac{\text{\rm ad}^ng}{n!}
\]
which defines an inner automorphism converges for all $g\in G$.
Considering inner automorphisms of a relatively free Lie algebra $F_m(G)$,
the first problem arising is that the formal power series defining
the inner automorphism has to be well defined. This means that
the operator $\text{\rm ad}w$, $w\in F_m(G)$, has to be locally nilpotent.
In many important cases
$\text{\rm ad}w$ is not locally nilpotent for some $w\in F_m(G)$.
Hence we have two possibilities to study inner automorphisms:

\noindent (1) to restrict the consideration to locally nilpotent
derivations $\text{\rm ad}w$,

\noindent or

\noindent
(2) to consider nilpotent relatively free algebras
\[
F_m(G)/F_m^{c+1}(G)=L_m/(I(G)+L_m^{c+1})
\]
when $\exp(\text{\rm ad}w)$
is well defined for all $w\in F_m(G)/F_m^{c+1}(G)$.

The well known formula $e^xe^y=e^{x+y}$ from Calculus is not more true
if we consider non-commuting variables $x$ and $y$.
If $x$ and $y$ are the generators of the free associative algebra $A=K\langle x,y\rangle$,
we may consider the completion $\widehat{A}$ with
respect to the formal power series topology.
Then the classical Baker-Campbell-Hausdorff formula gives the solution $z$ of the equation
$e^z=e^xe^y$ in $\widehat{A}$.
It is a formal power series in $\widehat{A}$
and its homogeneous components are Lie elements, i.e., $z$ belongs to the completion $\widehat{L_2}$
of the free Lie algebra $L_2$ generated by $x$ and $y$.
Similarly, if we consider the relatively free algebra $F_m=F_m(G)$, it has a completion
$\widehat{F_m}$ with respect to the formal power series topology.
The composition of two inner automorphisms of $\widehat{F_m}$
is an inner automorphism obtained by the Baker-Campbell-Hausdorff formula.

In the former case (1), when we consider only locally nilpotent
derivations of the algebra $F_m=F_m(G)$, let $\text{\rm ad}u$ and
$\text{\rm ad}v$ be locally nilpotent for some $u,v\in F_m(G)$. It
is not clear a priori whether the solution $w\in \widehat{F_m}$ of
the equation $\exp(\text{\rm ad}u)\exp(\text{\rm ad}v)=\exp(\text{\rm ad}w)$ belongs to $F_m$. If this
happens for all $u,v\in F_m$ with $\text{\rm ad}u,\text{\rm ad}v$
locally nilpotent, then the product of two inner automorphisms of
$F_m$ is inner again and the group of inner automorphisms
$\text{\rm Inn}(F_m)$ is well defined. For example, if
$F_m=L_m/L_m''$ is the free metabelian Lie algebra, then
$\text{ad}u$ is locally nilpotent for $u\in F_m$ if and only if
$u$ belongs to the commutator ideal $F_m'$ and
\[
\exp(\text{ad}u)\exp(\text{ad}v)=\exp(\text{ad}(u+v)),\quad u,v\in F_m'.
\]
It is easy to construct an example of a graded associative algebra $R$
and two elements $u,v\in R$ such that $u^k=v^n=0$ for some positive integers
$k$ and $n$ (which guarantees that $e^u,e^v\in R$) and the equation $e^ue^v=e^w$
has a solution $w$ in $\widehat{R}$ which does not belong to $R$,
see Example \ref{counterexample to product of exponentials}.
This gives also an example of a Lie algebra $H$ which is graded as an algebra, with well defined inner automorphisms
$\exp(\text{ad}u)$ and $\exp(\text{ad}v)$, $u,v\in H$, and such that the solution $w$
of the equation $\exp(\text{ad}u)\exp(\text{ad}v)=\exp(\text{ad}w)$ belongs to $\widehat{H}$
and not to $H$, see Example \ref{counterexample to product of inner autos}.
But we do not know any example of a relatively free algebra $F_m(G)$ with the property
that the product of inner automorphisms is not an inner automorphism of the algebra.
(By the Baker-Campbell-Hausdorff formula the product of inner automorphisms is inner in the completion
$\widehat{F_m}$ of $F_m=F_m(G)$.)

In the latter case (2) of nilpotent relatively free algebras $F_m(G)$, it is more convenient to work
in the completion $\widehat{F_m}$ of $F_m=F_m(G)$ which we prefer to do.
Then the results on $\text{\rm Inn}(F_m/F_m^{c+1})$
can be obtained immediately from
the corresponding results on $\text{\rm Inn}(\widehat{F_m})$
taking into account only the first
$c$ summands in the formal power series describing inner automorphisms.

Kofinas and Papistas \cite{KP} found a description of the automorphism group
of relatively free nilpotent Lie algebras $F_m$ over $\mathbb Q$ in terms of
the Baker-Campbell-Hausdorff formula. This allowed them to construct
generating sets for $\text{Aut}(F_m)$ for several varieties of nilpotent Lie algebras.

In the completion $\widehat{F_2}$ of a concrete relatively free Lie algebra $F_2=F_2(G)$
the Baker-Campbell-Hausdorff series may have a simpler form than in $\widehat{A}$.
For example, Gerritzen \cite{G} and later Kurlin \cite{K1, K2}
found an expression of the series in the completion of the free metabelian Lie algebra
$L_2/L_2''$. In our recent paper \cite{DF2}, see also \cite{DF1}, we have used their results to describe
the inner and outer automorphisms of the free metabelian nilpotent Lie algebras
$(L_m/L_m'')/(L_m/L_m'')^{c+1}$ for any $c$ and $m$.

Going back to the classics, Baker \cite{B} evaluated the Baker-Campbell-Hausdorff series
on several finite dimensional Lie algebras given in their adjoint representations,
including the three-dimensional simple Lie algebra.
In the present paper we study inner automorphisms of the completion of relatively free algebras
of the variety of algebras generated by the three-dimensional simple Lie algebra $G_3$.
If the field $K$ is algebraically closed, then $G_3$ is
isomorphic to the Lie algebra $sl_2(K)$ of traceless $2\times 2$ matrices.
This implies that $I(G_3)=I(sl_2(K))$ for any field $K$ of characteristic 0, see the comments in the end of Section
\ref{second section}.

In the theory of algebras with polynomial identities there is a standard generic construction which realizes
$F_m(G)$, $\text{\rm dim}(G)<\infty$. Fixing a basis of $G$, the algebra $F_m(G)$ is isomorphic to a subalgebra
of the algebra with elements which are obtained by replacing the coordinates of the elements of $G$ with
polynomials in $m\cdot\text{\rm dim}(G)$ variables. Our first result translates the result of Baker \cite{B}
in terms of this generic construction. We present a
multiplication rule for the inner automorphisms of the completion $\widehat{F_m}$
of $F_m=F_m(G_3)$ with respect to the formal power series topology.

The most important generic realization of
the relatively free algebra $F_m(G_3)=F_m(sl_2(K))$ is as the Lie algebra
generated by $m$ generic traceless $2\times 2$ matrices. This allows to apply
the approach of Baker \cite{B} and to obtain a formula
for the composition of inner automorphisms of $F_m(sl_2(K))$.
Although the structure of $F_m(sl_2(K))$ is known for all $m\geq 2$,
we consider the case $m=2$ only. The first reason for this restriction is that the structure of the algebra $F_m(sl_2(K))$
in the case $m>2$ is more complicated than for $m=2$. The second reason is that
the Lie algebra $F_2(sl_2(K))$ is naturally embedded into an associative algebra which is a free module
over the polynomial algebra in three variables and the commutator ideal of $F_2(sl_2(K))$
is a free submodule which allows to use methods of linear algebra.

We work in the completion $\widehat{W}$ of
the associative algebra $W$ generated by two generic
traceless $2\times 2$ matrices $x=(x_{ij})$ and $y=(y_{ij})$, where
 $x_{ij}$, $y_{ij}$, $(i,j)=(1,1),(1,2),(2,1)$, are algebraically independent
commuting variables, $x_{22}=-x_{11}$, $y_{22}=-y_{11}$. Let $L$ be
the Lie subalgebra of $W$ generated by $x$ and $y$. Then $L\cong F_2(sl_2(K))$.
We have obtained the complete description of the group of inner automorphisms
of the associative algebra $\widehat{W}$.
We give a lemma to recognize whether an element in $W$ belongs to $L$ and if the answer
is affirmative, an algorithm to write the element as a linear combination of commutators.
This leads to the description of the group of inner automorphisms
of the Lie algebra $\widehat{L}$: the multiplication rule for
$\text{\rm Inn}(\widehat{L})$ and the associated matrices of inner automorphisms of $\widehat{L}$.
Finally, factorizing the algebra $F_m(sl_2(K))$
modulo $F_m^{c+1}(sl_2(K))$ we derive the corresponding results
for the relatively free algebras in the variety $\text{var}(sl_2(K))\cap{\mathfrak N}_c$
of the nilpotent of class $\leq c$
Lie algebras in $\text{var}(sl_2(K))$.

It would be interesting to have a similar description for $\text{Inn}(\widehat{F_m})$
and $\text{Inn}(F_m/F_m^{c+1})$ also for $m>2$, where $F_m=F_m(sl_2(K))$,
but it seems that the answer will be quite technical.

Recently the second named author of the present paper \cite{F} has described the group
$\text{\rm Out}(\widehat{L})=\text{\rm Aut}(\widehat{L})/\text{\rm Inn}(\widehat{L})$
of outer automorphisms of $\widehat{L}$, where $\text{\rm Aut}(\widehat{L})$ is the group
of continuous automorphisms of $\widehat{L}$. This gives immediately the description of the group
$\text{\rm Out}(L/L^{c+1})$ of outer automorphisms of $L/L^{c+1}$.

The results of the present paper have been announced without proofs in \cite{DF1}.

\section{Preliminaries}\label{first section}

We fix a field $K$ of characteristic 0 and the associative algebra $W$ generated
by two generic traceless $2\times 2$ matrices
\[
x=\left(
\begin{array}{cc}
x_{11}&x_{12}\\
x_{21}&-x_{11}
\end{array}
\right),\quad
y=\left(
\begin{array}{cc}
y_{11}&y_{12}\\
y_{21}&-y_{11}
\end{array}
\right),
\]
where $x_{ij}$, $y_{ij}$, $(i,j)=(1,1),(1,2),(2,1)$, are algebraically independent
commuting variables. We assume that $W$ is a subalgebra of the $2\times 2$ matrix algebra
$M_2(K[x_{ij},y_{ij}])$ and identify the polynomial $f\in K[x_{ij},y_{ij}]$
with the scalar matrix with entries $f$ on the diagonal. In particular, for any matrix $z\in W$
we assume that the trace $\text{tr}(z)$ belongs to the centre of $M_2(K[x_{ij},y_{ij}])$.
Let  $L$ be the Lie subalgebra of $W$ generated by $x$ and $y$. This is the
smallest subspace of the vector space $W$ containing  $x$ and $y$
and closed with respect to the Lie multiplication
\[
[z_1,z_2]=z_1\text{\rm ad}z_2=z_1z_2-z_2z_1, \quad z_1,z_2\in L.
\]
Similarly we define the associative algebra $W_m$ generated
by $m\geq 2$ generic traceless $2\times 2$ matrices.
We assume that all commutators are left normed, i.e.,
\[
[z_1,\ldots,z_{n-1},z_n]=[[z_1,\ldots,z_{n-1}],z_n],\quad n=3,4,\ldots.
\]

The following results give the description of the algebras $W_m$, $W=W_2$ and $L$ and some equalities
in $W$.

\begin{theorem}\label{description of W}
Let $W_m$, $W$ and $L$ be as above. Then:

{\rm (i) (Razmyslov \cite{R})}
The algebra of generic traceless matrices $W_m$ is isomorphic to the factor-algebra
$K\langle x_1,\ldots,x_m\rangle/I(M_2(K),sl_2(K))$ of the free associative algebra $K\langle x_1,\ldots,x_m\rangle$,
where the ideal $I(M_2(K),sl_2(K))$ of the weak polynomial identities in $m$ variables for the pair
$(M_2(K),sl_2(K))$ consists of all polynomials from $K\langle x_1,\ldots,x_m\rangle$
which vanish on $sl_2(K)$ considered as a subset of $M_2(K)$. As a weak $T$-ideal
$I(M_2(K),sl_2(K))$ is generated by the
weak polynomial identity $[x_1^2,x_2]=0$. The Lie subalgebra of $W_m$ generated by the $m$ generic traceless matrices
is isomorphic to the relatively free algebra $F_m(sl_2(K))$
in the variety of Lie algebras generated by $sl_2(K)$.

{\rm (ii) (Drensky, Koshlukov \cite{DK}, see also the comments in \cite{D2}
and Koshlukov \cite{Ko1, Ko2} for the case of positive characteristic)}
The algebra $W_m$ has the presentation
\[
W_m\cong K\langle x_1,\ldots,x_m\mid [x_i^2,x_j]=s_4(x_{i_1},x_{i_2},x_{i_3},x_{i_4})=0\rangle,
\]
where $i,j,i_k=1,\ldots,m$, $i\not= j$, $i_1<i_2<i_3<i_4$, and
\[
s_4(x_1,x_2,x_3,x_4)=\sum_{\sigma\in S_4}\text{\rm sign}(\sigma)x_{\sigma(1)}x_{\sigma(2)}x_{\sigma(3)}x_{\sigma(4)}
\]
is the standard polynomial of degree $4$. In particular,
\[
W\cong K\langle x_1,x_2\mid [x_1^2,x_2]=[x_2^2,x_1]=0\rangle.
\]

{\rm (iii) (see e.g. Le Bruyn \cite{LB})}
The centre of $W$ is generated by
\[
t=\text{\rm tr}(x^2), \quad u=\text{\rm tr}(y^2), \quad v=\text{\rm tr}(xy).
\]
The elements $t,u,v$ are algebraically independent
and $W$ is a free $K[t,u,v]$-module with free generators $1,x,y,[x,y]$.

{\rm (iv) (see e.g. Drensky and Gupta \cite{DG})} For $k\geq 1$ the following equalities hold in $W$:
\[
x^2=\frac{t}{2};\quad y^2=\frac{u}{2};\quad xy+yx=v;\quad [x,y]^2=v^2-tu;
\]
\[
y\text{\rm ad}^{2k}x=2^kt^{k-1}(-vx+ty);\quad y\text{\rm ad}^{2k+1}x=2^kt^{k}[y,x];
\]
\[
x\text{\rm ad}^{2k}y=2^ku^{k-1}(ux-vy);\quad x\text{\rm ad}^{2k+1}y=2^ku^{k-1}[x,y].
\]
\end{theorem}

Theorem \ref{description of W} (iii) and (iv) gives immediately that
$L$ is embedded into the free $K[t,u,v]$-module with free generators $x,y,[x,y]$.
The next lemma gives the precise description of the Lie elements in $W$.
It also provides an algorithm how to express in Lie form the elements of $L$ given
as elements of the free $K[t,u,v]$-module with basis $x,y,[x,y]$.

\begin{lemma}\label{Lie elements in W}
{\rm (i)} The commutator ideal $L'$ of $L\cong F_2(sl_2(K))$ is a free $K[t,u,v]$-module of rank $3$, with free generators
\[
xv-yt,\quad xu-yv,\quad [x,y].
\]

{\rm (ii)} The elements of
\[
L'=(xv-yt)K[t,u,v]\oplus (xu-yv)K[t,u,v]\oplus [x,y]K[t,u,v]
\]
can be expressed in Lie form using the identities
\[
2^{a+b+c+1}(xv-yt)t^au^bv^c=[x,y,y](\text{\rm ad}y)^{2b-1}(\text{\rm ad}x)^{2a+1}(\text{\rm ad}y\text{\rm ad}x)^c,\quad b>0,
\]
\[
2^{a+c+1}(xv-yt)t^av^c=[x,y,x](\text{\rm ad}x)^{2a}(\text{\rm ad}y\text{\rm ad}x)^c,
\]
\[
2^{a+b+c+1}(xu-yv)t^au^bv^c=[x,y,x](\text{\rm ad}x)^{2a-1}(\text{\rm ad}y)^{2b+1}(\text{\rm ad}x\text{\rm ad}y)^c,\quad a>0,
\]
\[
2^{b+c+1}(xu-yv)u^bv^c=[x,y,y](\text{\rm ad}y)^{2b}(\text{\rm ad}x\text{\rm ad}y)^c,
\]
\[
2^{a+b+c}[x,y]t^au^bv^c=[x,y](\text{\rm ad}x)^{2a}(\text{\rm ad}y)^{2b}(\text{\rm ad}x\text{\rm ad}y)^c.
\]
\end{lemma}

\begin{proof}
(i) We make induction on the degree of the commutators using that every element in a Lie algebra
is a linear combination of left normed commutators. The following equalities
\[
[x,y,x]=-[y,x,x]=2(xv-yt),\quad [x,y,y]=-[y,x,y]=2(xu-yv),
\]
\[
[x,y,x,x]=2[x,y]t,\quad [x,y,x,y]=[x,y,y,x]=2[x,y]v,\quad [x,y,y,y]=2[x,y]u
\]
\[
[xv-yt,x,x]=2(xv-yt)t,\quad [xv-yt,x,y]=2(xu-yv)t,
\]
\[
[xv-yt,y,x]=2(xv-yt)v,\quad [xv-yt,y,y]=2(xu-yv)v,
\]
\[
[xu-yv,x,x]=2(xv-yt)v,\quad [xu-yv,x,y]=2(xu-yv)v,
\]
\[
[xu-yv,y,x]=2(xv-yt)u,\quad [xu-yv,y,y]=2(xu-yv)u
\]
show that $L'$ coincides with the set of elements of $W$ in the form
\[
(xv-yt)f+(xu-yv)g+[x,y]h,\quad f,g,h\in K[t,u,v].
\]
This means that $L'$ is the $K[t,u,v]$-module generated by
$xv-yt,xu-yv,[x,y]$.
If such an element is equal to 0 then we have
\[
x(vf+ug)+y(-tf-vg)+[x,y]h=0
\]
in the free $K[t,u,v]$-module $xK[t,u,v]+yK[t,u,v]+[x,y]K[t,u,v]$. Hence
\[
vf+ug=0,\quad tf+vg=0,\quad h=0.
\]
Since the determinant $v^2-tu$ of the system of the first two equations
(with unknowns $f,g$) is different from 0 we have $f=g=0$.

(ii) The proof follows from by easy inductions using the equations in the proof of (i).
\end{proof}

The relatively free Lie algebra $F_m(sl_2(K))$, $m\geq 2$, has a nice description
in terms of representation theory of the general linear group $GL_m(K)$, see
\cite[Exercise 12.6.10, p. 245]{D1}. But for $m>2$ the center of the algebra $W_m$
has a lot of defining relations, see \cite{D2} and there is no analogue of Lemma \ref{Lie elements in W}.
This makes the explicit computations in $F_m(sl_2(K))$ more complicated.

Let $R$ be a (not necessarily associative) graded $K$-algebra,
\[
R=\bigoplus_{n\geq 0}R_{(n)}=R_{(0)}\oplus R_{(1)}\oplus R_{(2)}\oplus\cdots,
\]
where $R_{(n)}$ is the homogeneous component of degree $n$ in $R$,
and $R_{(0)}=0$ or $R_{(0)}=K$.
We consider the {\it formal power series topology} on $R$ induced by the filtration
\[
\omega^0(R)\supseteq \omega^1(R)\supseteq \omega^2(R)\supseteq\cdots,\quad
\omega^n(R)=\bigoplus_{k\geq n}R_{(k)},\quad n=0,1,2,\ldots,
\]
where $\omega(R)=R$ if $R_{(0)}=0$, and $\omega(R)$ is the augmentation ideal of $R$ when $R_{(0)}=K$.
This is the topology in which the sets
\[
r+\omega^n(R),\quad r\in R,\quad n\geq 0,
\]
form a basis for the open sets.
We shall denote by $\widehat{R}$ the completion of $R$ with respect to the formal power series topology
and shall identify it
with the Cartesian sum $\widehat \bigoplus_{n\geq 0}R_{(n)}$.
The elements $f\in \widehat{R}$ are formal power series
\[
f=f_0+f_1+f_2+\cdots,\quad f_n\in R_{(n)},\quad n=0,1,2,\ldots,
\]
A sequence
\[
f^{(k)}=f_{k0}+f_{k1}+f_{k2}+\cdots,\quad  k=1,2,\ldots,
\]
where $f_{kn}\in R_{(n)}$,
converges to $f=f_0+f_1+f_2+\cdots$, where $f_n\in R_{(n)}$, if for every $n_0$ there exists a $k_0$ such that
$f_{kn}=f_n$ for all $n< n_0$ and all $k\geq k_0$, i.e., for all sufficiently large $k$ the first $n_0$ terms
of the formal power series $f^{(k)}$ are the same as the first $n_0$ terms of $f$.

Let $F_m=F_m(G)$ be a relatively free algebra freely generated by $x_1,\ldots,x_m$. Then $F_m$ is graded
and the $n$th homogeneous component is spanned by all commutators $[x_{i_1},\ldots,x_{i_n}]$ of length $n$.
Hence the elements of $\widehat{F_m}$ are formal series of commutators. Since $[F_m^n,u]=F_m^n\text{ad}u\subset F_m^{n+1}$
for any $u\in F_m$, we derive that the inner automorphisms $\exp(\text{ad}u)$ of $\widehat{F_m}$ are continuous automorphisms.

Let $W_{(n)}$ be the subspace of $W$ spanned by all monomials
of total degree $n$ in $x,y$.
The elements $f\in \widehat{W}$ are formal power series
\[
f=f_0+f_1+f_2+\cdots,\quad f_n\in W_{(n)},\quad n=0,1,2,\ldots,
\]
and $\widehat{W}$ is a free $K[[t,u,v]]$-module with free generators $1,x,y,[x,y]$,
where $K[[t,u,v]]$ is the algebra of formal power series in the variables $t,u,v$.
Since $\widehat{L}$ is embedded canonically into
$\widehat{W}$, Lemma \ref{Lie elements in W} gives that
$(\widehat{L})'$ is a free $K[[t,u,v]]$-module with free generators $xv-yt,xu-yv,[x,y]$ and
\[
\widehat{L}=\{\alpha x+\beta y+a(xv-yt)+b(xu-yv)+c[x,y]\mid \alpha,\beta\in K, a,b,c\in K[[t,u,v]]\}.
\]

The Baker-Campbell-Hausdorff formula gives the solution $z$ of the equation
\[
e^z=e^xe^y
\]
for non-commuting $x$ and $y$, see e.g. \cite{Bo} and \cite{Se}.
If $x,y$ are the generators of the free associative algebra $A=K\langle x,y\rangle$, then
\[
z=x+y+\frac{[x,y]}{2}-\frac{[x,y,x]}{12}+\frac{[x,y,y]}{12}-\frac{[x,y,x,y]}{24}+\cdots
\]
is a formal power series in the completion $\widehat{A}$. The homogeneous components of $z$ are Lie elements
and $z\in \widehat{L_2}$ where $L_2$ is canonically embedded into $A$.

The composition of two inner automorphisms in $\text{\rm Inn}(\widehat{W})$ is also an inner automorphism
which can be obtained by the Baker-Campbell-Hausdorff formula.
Hence, studying the inner automorphisms of  $\widehat{L}$,
it is convenient to work in  $\widehat{W}$ and to study
the group of its inner automorphisms.

If $\delta$ is an endomorphism of the free $K[[t,u,v]]$-submodule of $\widehat{W}$ with basis $\{x,y,[x,y]\}$,
then we denote by $\text{\rm M}(\delta)$ the {\it associated} matrix of $\delta$ with respect to this basis. If
\[
\delta(x)=\sigma_{11}x+\sigma_{21}y+\sigma_{31}[x,y],
\]
\[
\delta(y)=\sigma_{12}x+\sigma_{22}y+\sigma_{32}[x,y],
\]
\[
\delta([x,y])=\sigma_{13}x+\sigma_{23}y+\sigma_{33}[x,y],
\]
$\sigma_{ij}\in K[[t,u,v]]$, then
\[
\text{\rm M}(\delta)=\left(
\begin{array}{ccc}
\sigma_{11}&\sigma_{12}&\sigma_{13}\\
\sigma_{21}&\sigma_{22}&\sigma_{23}\\
\sigma_{31}&\sigma_{32}&\sigma_{33}
\end{array}
\right).
\]
Clearly $M(\delta)$ behaves as a matrix of a usual linear operator. In particular,
\[
\text{\rm M}(\delta_1\delta_2)=\text{\rm M}(\delta_1)\text{\rm M}(\delta_2).
\]
Since the derivation $\text{\rm ad}X$, $X\in \widehat{W}$, acts trivially on the centre of $\widehat{W}$,
it is an endomorphism of $\widehat{W}$ as a $K[[t,u,v]]$-module.
Its restriction on the submodule generated by $x,y,[x,y]$ satisfies the above conditions.
Hence the matrix $\text{\rm M}(\text{\rm ad}X)$ is well defined, and similarly for the matrix
$\text{\rm M}(\exp(\text{\rm ad}X))$.

\section{The work of Baker and its consequences}\label{second section}

In the sequel we shall use calculations of Baker \cite{B}.
Let $G_3$ be the three-dimensional complex simple Lie algebra with the basis $\{p_1,p_2,p_3\}$
and multiplication
\[
[p_1,p_2]=p_1, \quad [p_1,p_3]=2p_2, \quad [p_2,p_3]=p_3,
\]
and let $X=x_1p_1+x_2p_2+x_3p_3\in G_3$, $x_1,x_2,x_3\in\mathbb C$.
We denote respectively by $P(X)$ and $Q(\exp(\text{\rm ad}X))$ the matrices of the linear operators
$\text{\rm ad}X$ and $\exp(\text{\rm ad}X)$ with respect to the basis
$\{p_1,p_2,p_3\}$. Clearly,
\[
P(X)=\left(
\begin{array}{cccc}
x_2&-x_1&0\\
2x_3&0&-2x_1\\
0&x_3&-x_2
\end{array}
\right).
\]

\begin{theorem}\label{Formula of Baker}{\rm(Baker \cite{B})}
{\rm (i)} In the above notation, the matrices $P(X)$ and $Q(\exp(\text{\rm ad}X))$
satisfy
\[
P^3(X)=g(X)P(X), \quad g(X)=x_2^2-4x_1x_3,
\]
\begin{align}
Q(\exp(\text{\rm ad}X))&=I_3+P(X)\left(1+\frac{g(X)}{3!}+
\frac{g^2(X)}{5!}+\frac{g^3(X)}{7!}+\cdots\right)+\nonumber\\
&\quad+P^2(X)\left(\frac{1}{2}+\frac{g(X)}{4!}+\frac{g^2(X)}{6!}+
\cdots\right)\nonumber\\
&=I_3+A(X)P(X)+B(X)P^2(X) \nonumber,
\end{align}
where $I_3$ is the $3\times 3$ identity matrix and
\[
A(X)=\frac{\sinh (\sqrt{g(X)})}{\sqrt{g(X)}}, \quad B(X)=\frac{\cosh (\sqrt{g(X)})-1}{g(X)}.
\]

{\rm (ii)} Let $X,Y\in G_3$ and let $Z=z_1p_1+z_2p_2+z_3p_3\in G_3$, $z_1,z_2,z_3\in\mathbb C$,
be such that
\[
\exp(\text{\rm ad}Z)=\exp(\text{\rm ad}X)\exp(\text{\rm ad}Y).
\]
If $Q(\exp(\text{\rm ad}Z))=(\sigma_{ij})$, then
\[
z_1=\frac{M_1}{2A},\quad z_2=\frac{M_2}{2A},\quad z_3=\frac{M_3}{2A},
\]
where
\[
M_1=-\sigma_{12}-\frac{1}{2}\sigma_{23},\quad M_2=\sigma_{11}-\sigma_{33},
\quad M_3=\frac{1}{2}\sigma_{21}+\sigma_{32},
\]
\[
A=\frac{\sinh (\sqrt{g(Z)})}{\sqrt{g(Z)}},\quad B=\frac{\cosh (\sqrt{g(Z)})-1}{g(Z)},
\]
\[
g(Z)=z_2^2-4z_1z_3=\left(\log \frac {1-M+\sqrt{1-2M}}{M}\right)^2,
\]
\[
M=\frac {2\sigma_{11}+2\sigma_{33}-4}{M_2^2-2M_1M_3}
=\frac {2\sigma_{13}}{M_1^2}
=\frac {2\sigma_{31}}{M_3^2}
=\frac {\sigma_{21}-2\sigma_{32}}{M_2M_3}
\]
\[
=\frac {\sigma_{23}-2\sigma_{12}}{M_1M_2}
=\frac {1-2\sigma_{22}}{M_1M_3}
=\frac {B}{A^2}.
\]
\end{theorem}

Now we consider the algebra $G_3$ from the consideration of Baker over an arbitrary field
$K$ of characteristic 0. The following generic construction is well known. Let $m\geq 2$ and let
\[
K[x_{ij}]=K[x_{ij}\mid i=1,2,3,\quad j=1,\ldots,m]
\]
be the polynomial algebra in $3m$ variables.
We consider the tensor product $K[x_{ij}]\otimes_KG_3$ which is a Lie $K$-algebra.
It is a free $K[x_{ij}]$-module with basis $\{p_1,p_2,p_3\}$. We shall omit the symbol $\otimes$ for the tensor product
in the elements of $K[x_{ij}]\otimes_KG_3$. We fix the elements
\[
X_j=x_{1j}p_1+x_{2j}p_2+x_{3j}p_3,\quad j=1,\ldots,m.
\]
The Lie subalgebra generated by $X_1,\ldots,X_m$ in $K[x_{ij}]\otimes_KG_3$
is isomorphic to the relatively free algebra $F_m(G_3)=F_m(\text{\rm var}G_3)$
and we shall identify both algebras. Clearly, the completion of $F_m(G_3)$
with respect to the formal power series topology is canonically embedded into
$K[[x_{ij}]]\otimes_KG_3$. For an arbitrary
\[
V=v_1p_1+v_2p_2+v_3p_3,\quad v_i\in K[[x_{ij}]],
\]
the operators $\text{\rm ad}V$ and $\exp(\text{\rm ad}V)$ of
$K[[x_{ij}]]\otimes_KG_3$ act as endomorphisms of the free $K[[x_{ij}]]$-module
$K[[x_{ij}]]\otimes_KG_3$.
We denote respectively by $P(V)$ and $Q(\exp(\text{\rm ad}V))$ the matrices of these endomorphisms
with respect to the basis
$\{p_1,p_2,p_3\}$.

Our first result is the following corollary which translates the result of Baker given
in Theorem \ref{Formula of Baker} in terms of the above generic construction.

\begin{corollary}\label{Formula of Baker generic version}
If $X,Y,Z\in K[[x_{ij}]]\otimes_KG_3$,
\[
X=x_1p_1+x_2p_2+x_3p_3,\quad Y=y_1p_1+y_2p_2+y_3p_3,\quad
Z=z_1p_1+z_2p_2+z_3p_3,
\]
$x_i,y_i,z_i\in K[[x_{ij}]]$, are such that
\[
\exp(\text{\rm ad}Z)=\exp(\text{\rm ad}X)\exp(\text{\rm ad}Y),
\]
then the related matrices $P(V)$ and $Q(\exp(\text{\rm ad}V))$,
$V=X,Y,Z$, satisfy the relations of Theorem \ref{Formula of Baker}.
\end{corollary}

\begin{proof}
Obviously, $Z\in K[[x_{ij}]]\otimes_KG_3$ is uniquely determined from the condition
$\exp(\text{\rm ad}Z)=\exp(\text{\rm ad}X)\exp(\text{\rm ad}Y)$. The computations of Baker in
Theorem \ref{Formula of Baker} are performed for power series which converge evaluated in $\mathbb C$.
They hold also over $K[[x_{ij}]]$ if the corresponding expressions have sense there. Working in
$K[[w]]$, we have that
\[
\frac{\sinh(\sqrt{w})}{\sqrt{w}}=1+\frac{w}{3!}+\frac{w^2}{5!}+\frac{w^3}{7!}+\cdots,
\quad
\frac{\cosh(\sqrt{w})-1}{w}=\frac{1}{2!}+\frac{w}{4!}+\frac{w^2}{6!}+\cdots.
\]
These expressions are well defined in $K[[x_{ij}]]$ for $w=v_2^2-4v_1v_3$
when $v_1,v_2,v_3\in K[[x_{ij}]]$ are formal power series without constant term.
Since the constant term of $M=B/A^2$ in Theorem \ref{Formula of Baker} is equal
to $1/2$, we verify directly that the logarithm in the expression of $g(Z)$ has no constant term,
and this implies that all formal expressions have sense in $K[[x_{ij}]]$.
\end{proof}

The three-dimensional Lie algebra $G_3$ is simple. Extending the field $K$ to its algebraic closure $\overline{K}$,
we obtain an algebra isomorphic to $sl_2(\overline{K})$. It is well known that
if two $K$-algebras $G_1$ and $G_2$ become isomorphic over an extension of the infinite filed $K$,
then $G_1$ and $G_2$ have the same polynomial identities.
Hence the relatively free algebras $F_m(G_3)$ and $F_m(sl_2(K))$ are isomorphic.
Corollary \ref{Formula of Baker generic version} provides the multiplication formula for arbitrary inner automorphisms
of the completion $\widehat{F_m}$ of $F_m=F_m(G_3)$ in terms of the generic elements
$X_1,\ldots,X_m$. But it does not give the expression of $\exp(\text{\rm ad}Z))$ as a formal power series of Lie commutators.
In the next section we shall mimic the computations of Baker for the Lie algebra generated by two generic traceless
$2\times 2$ matrices and shall show how to find the Lie expression of the Baker-Campbell-Hausdorff formula
modulo the polynomial identities of $sl_2(K)$.

\section{Inner automorphisms of the Lie algebra of two generic matrices}\label{third section}

In this section we find the explicit form of the
associated matrix of the inner automorphisms of $\widehat{W}$ and
a multiplication rule for $\text{\rm Inn}(\widehat W)$.
Then we
transfer the obtained results to the algebra
$W/\omega(W)^{c+1}$ and obtain the description of
$\text{\rm Inn}(W/\omega(W)^{c+1})$.
Applying Lemma \ref{Lie elements in W} we obtain immediately the corresponding results
for the Lie algebras $\widehat{L}$ and $L/L^{c+1}$.

Let $\text{\rm Inn}(\widehat W)$ denote the set of all
inner automorphisms of $\widehat W$ which are of the form
$\exp(\text{\rm ad}X)$, $X\in \widehat W$.

As we already discussed, since $\widehat W$ is a $K[[t,u,v]]$-module with the generators $1,x,y,[x,y]$ and
$\text{\rm ad}X$ acts trivially on $1$ it is sufficient to know the action of inner automorphisms on only $x,y,[x,y]$.

\begin{theorem}\label{Formula for inner}
Let $X=ax+by+c[x,y]$, $a,b,c\in K[[t,u,v]]$, be an element in $\widehat W$ and let
\[
\text{\rm M}(\text{\rm ad}X)=\left(
\begin{array}{cccc}
-2cv&-2cu&2(av+bu)\\
2ct&2cv&-2(at+bv)\\
b&-a&0
\end{array}
\right)
\]
be the associated matrix of $\text{\rm ad}X$. Then the associated matrix of $\exp(\text{\rm ad}X)$ is of the form
\[
\text{\rm M}(\exp(\text{\rm ad}X))=I_3+A(X)\text{\rm M}(\text{\rm ad}X)+B(X)\text{\rm M}^2(\text{\rm ad}X),
\]
where
\[
A(X)=\frac{\sinh (\sqrt{g(X)})}{\sqrt{g(X)}}, \quad B(X)=\frac{\cosh (\sqrt{g(X)})-1}{g(X)},
\]
\[
g(X)=2(a^2t+2abv+b^2u+2c^2(v^2-tu)).
\]
\end{theorem}

\begin{proof}
From Theorem \ref{description of W} (iv) we know that
\[
[x,[x,y]]=2(-xv+yt);\quad [y,[x,y]]=2(-xu+yv).
\]
By easy calculations we obtain that
\[
x\text{\rm ad}X=-2cvx+2cty+b[x,y],
\]
\[
y\text{\rm ad}X=-2cux+2cvy-a[x,y],
\]
\[
[x,y]\text{\rm ad}X=2(av+bu)x-2(at+bv)y,
\]
\[
\text{\rm M}(\text{\rm ad}X)=\left(
\begin{array}{cccc}
-2cv&-2cu&2(av+bu)\\
2ct&2cv&-2(at+bv)\\
b&-a&0
\end{array}
\right),
\]
\[
\text{\rm M}^2(\text{\rm ad}X)=\left(
\begin{array}{cccc}
4c^2w+2b(av+bu)&-2a(av+bu)&-4acw\\
-2b(at+bv)&4c^2w+2a(at+bv)&-4bcw\\
-2c(at+bv)&-2c(av+bu)&2a(at+bv)+2b(av+bu)
\end{array}
\right),
\]
$w=v^2-tu$.
Calculating $\text{\rm M}^3(\text{\rm ad}X)$ we obtain that
\[
\text{\rm M}^3(\text{\rm ad}X)=g(X)\text{\rm M}(\text{\rm ad}X),\quad
g(X)=2(a^2t+2abv+b^2u+2c^2(v^2-tu)).
\]
Following the steps in Theorem \ref{Formula of Baker} we obtain that
\[
\text{\rm M}(\exp(\text{\rm ad}X))=I_3+A(X)\text{\rm M}(\text{\rm ad}X)+B(X)\text{\rm M}^2(\text{\rm ad}X),
\]
\[
A(X)=\frac{\sinh (\sqrt{g(X)})}{\sqrt{g(X)}}, \quad B(X)=\frac{\cosh (\sqrt{g(X)})-1}{g(X)}.
\]
\end{proof}

\begin{theorem}\label{algorithm Inn}
If $Q=\text{\rm M}(\exp(\text{\rm ad}X))$
is the associated matrix of an inner automorphism for an element $X$ in $\widehat W$,
then the associated matrix $\text{\rm M}(\text{\rm ad}X)$ of the derivation $\text{\rm ad}X$ is
\[
\text{\rm M}(\text{\rm ad}X)=\frac{A(X)}{B(X)}(Q-I_3)-\frac{1}{2A(X)}(Q^2-I_3),
\]
where
\[
\cosh(\sqrt{g(X)})=\frac{1}{2}\left(\text{\rm tr}(Q)-1\right),
\]
\[
A(X)=\frac{\sinh(\sqrt{g(X)})}{\sqrt{g(X)}}, \quad B(X)=\frac{\cosh(\sqrt{g(X)})-1}{g(X)}.
\]
\end{theorem}

\begin{proof}
Since  $\text{\rm M}^3(\text{\rm ad}X)=g(X)\text{\rm M}(\text{\rm ad}X)$, the Cayley-Hamilton theorem gives that
\[
\text{\rm M}(\text{\rm ad}X)\sim
\left(
\begin{array}{cccc}
0&0&0\\
0&\sqrt{g(X)}&0\\
0&0&-\sqrt{g(X)}
\end{array}
\right),
\]
\[
Q=\text{\rm M}(\exp(\text{\rm ad}X))
\sim
\left(
\begin{array}{cccc}
1&0&0\\
0&\exp(\sqrt{g(X)})&0\\
0&0&\exp(-\sqrt{g(X)})
\end{array}
\right).
\]
Hence the trace of $Q$ is equal to $1+\exp(\sqrt{g(X)})+\exp(-\sqrt{g(X)})$ which gives the expression for
\[
\cosh(\sqrt{g(X)})=\frac{1}{2}\left(\exp(\sqrt{g(X)})+\exp(-\sqrt{g(X)})\right)
=\frac{1}{2}\left(\text{\rm tr}(Q)-1\right).
\]
So we have also
\[
A(X)=\frac{\sinh(\sqrt{g(X)})}{\sqrt{g(X)}}, \quad B(X)=\frac{\cosh(\sqrt{g(X)})-1}{g(X)}.
\]
We consider the expressions for $Q$ and $Q^2$
\[
Q=I_3+A(X)\text{\rm M}(\text{\rm ad}X)+B(X)\text{\rm M}^2(\text{\rm ad}X),
\]
\[
Q^2=I_3+2A(X)(1+B(X)g(X))\text{\rm M}(\exp(\text{\rm ad}X))
\]
\[
+(A^2(X)+B^2(X)g(X)+2B(X))\text{\rm M}^2(\exp(\text{\rm ad}X))
\]
as a linear system with unknowns $\text{\rm M}(\text{\rm ad}X)$ and $\text{\rm M}^2(\text{\rm ad}X)$.
Using the equality
\[
A^2(X)=B^2(X)g(X)+2B(X),
\]
the solution of the system gives
\begin{align}
\text{\rm M}(\text{\rm ad}X)&=\frac{-(A^2(X)+B^2(X)g(X)+2B(X))(Q-I_3)+B(X)(Q^2-I_3)}{-A(X)(A^2(X)+B^2(X)g(X)+2B(X))+2A(X)B(X)(1+B(X)g(X))}\nonumber\\
&=\frac{-2A^2(X)(Q-I_3)+B(X)(Q^2-I_3)}{-2A(X)B(X)}\nonumber.
\end{align}
\end{proof}

\begin{remark}
Theorem \ref{algorithm Inn} shows how to find $X\in \widehat{W}$ if we know the matrix
$Q=\text{\rm M}(\exp(\text{\rm ad}X))$. Since
$\cosh(\sqrt{g(X)})=(\text{\rm tr}(Q)-1)/2$ is of the form $1+h$ for some $h\in \omega(K[[t,u,v]])$, we use the formula
\[
\sqrt{g(X)}=\text{\rm arccosh}(1+h)=\log\left(1+h+\sqrt{h^2+2h}\right)
\]
which can be easily verified, and we obtain expressions for $g(X),A(X)$ and $B(X)$. Combined with
Theorem \ref{Formula for inner} we obtain also the multiplication rule for the group of inner
automorphisms of $\widehat W$: Given $\text{\rm ad}X$ and $\text{\rm ad}Y$, we calculate
consecutively $Q_1=\text{\rm M}(\exp(\text{\rm ad}X))$, $Q_2=\text{\rm M}(\exp(\text{\rm ad}Y))$,
their product $Q=\text{\rm M}(\exp(\text{\rm ad}Z))=Q_1Q_2$ and finally $Z$.
\end{remark}

Let us denote by $\omega$ the augmentation ideal of the polynomial algebra $K[t,u,v]$ consisting of the polynomials
without constant terms and let us denote its completion $\widehat{\omega}\subset K[[t,u,v]]$ with respect to the formal power formal series.
Since the elements $t=2x^2$, $u=2y^2$, $v=xy+yx$ are of even degree in $W$, the associated matrices of the automorphisms
of $\widehat{W}$ modulo $\widehat{\omega(W)}^{c+1}$, $c\geq 3$, contain the entries in the factor algebra $K[t,u,v]/\omega^{[(c+1)/2]}$.

As a consequence of our Theorems \ref{Formula for inner} and \ref{algorithm Inn}
for $\text{\rm Inn}(\widehat W)$ we
immediately obtain the description of the group of inner automorphisms of
$W/\omega(W)^{c+1}$. We shall give the results for the associated matrices only. The multiplication rule for
the group $\text{\rm Inn}(W/\omega(W)^{c+1})$ can be stated similarly.

\begin{corollary}
Let $X=ax+by+c[x,y]$, $a,b,c\in K[[t,u,v]]$, be an element in $\widehat{W}$ and let  $\text{\rm M}(\text{\rm ad}X)$ be the associated matrix
of $\text{\rm ad}X$. Then the associated matrix $\text{\rm M}(\exp(\text{\rm ad}X))$ of the inner automorphism
$\exp(\text{\rm ad}X)$ of $W/\omega(W)^{c+1}\cong \widehat{W}/\omega(\widehat{W})^{c+1}$ is of the form
\[
\text{\rm M}(\exp(\text{\rm ad}X))=I_3+A(X)\text{\rm M}(\text{\rm ad}X)+B(X)\text{\rm M}^2(\text{\rm ad}X)
\quad (\text{\rm mod }M_3(\omega^{[(c+1)/2]})),
\]
\[
A(X)=\frac{\sinh (\sqrt{g(X)})}{\sqrt{g(X)}}, \quad B(X)=\frac{\cosh (\sqrt{g(X)})-1}{g(X)},
\]
where
\[
\text{\rm M}(\text{\rm ad}X)=\left(
\begin{array}{cccc}
-2cv&-2cu&2(av+bu)\\
2ct&2cv&-2(at+bv)\\
b&-a&0
\end{array}
\right),
\]
\[
g(X)=2(a^2t+2abv+b^2u+2c^2(v^2-tu))
\]
and $M_3(\omega^{[(c+1)/2]})$ is the $3\times 3$ matrix algebra with entries from the $[(c+1)/2]$-th power of the
augmentation ideal of $K[t,u,v]$.
\end{corollary}

We shall complete the paper with two examples which are in the spirit of our considerations.
They are based on the simplest relations which guarantee that the corresponding
exponentials are well defined. We believe that the existence of such examples is folklorely
known although we were not able to find any references.

\begin{example} \label{counterexample to product of exponentials}
Let $R$ be the two-generated associative algebra with presentation
\[
R=K\langle u,v\mid u^2=v^2=0\rangle.
\]
Then
\[
e^u=1+u,\quad e^v=1+v
\]
are elements of $R$ but the equation $e^ue^v=e^w$ has no solution $w$ in $R$.
\end{example}

\begin{proof}
Since $R$ is a monomial algebra it has a basis as a $K$-vector space consisting of all non-commutative monomials
which do not contain as subwords $u^2$ and $v^2$.
The completion $\widehat{R}$ of $R$ with respect to the formal power series topology consists of formal power series
of such monomials. Consider the $2\times 2$ matrices
\[
U=\left(\begin{matrix}
0&a\\
0&0\\
\end{matrix}\right),
\quad
V=\left(\begin{matrix}
0&0\\
b&0\\
\end{matrix}\right),
\]
where $a,b$ are algebraically independent commuting variables.
Since $U^2=V^2=0$, the $K$-subalgebra $R_1$ of $M_2(K[a,b])$ generated by $U$ and $V$
is a homomorphic image of the algebra $R$. Algebras like $R_1$ appear in the paper by Belov \cite{Be}.
Tracing his considerations, it follows that
the algebras $R$ and $R_1$ are isomorphic. See also \cite{DSW} for a similar matrix realization of algebras
generated by two quadratic elements and direct proof of the isomorphism of $R$ and $R_1$.
We shall work in the algebra $R_1$ and its completion $\widehat{R_1}$ and
show that the solution $W$ of the equation $e^Ue^V=e^W$ in $\widehat{R_1}$ does not belong to $R_1$.
Direct computations give that
\[
e^U=\left(\begin{matrix}
1&a\\
0&1\\
\end{matrix}\right),
\quad
e^V=\left(\begin{matrix}
1&0\\
b&1\\
\end{matrix}\right),\quad
e^Ue^V=\left(\begin{matrix}
1+ab&a\\
b&1\\
\end{matrix}\right)=I+T=e^W,
\]
\[
I=\left(\begin{matrix}
1&0\\
0&1\\
\end{matrix}\right),\quad
T=\left(\begin{matrix}
ab&a\\
b&0\\
\end{matrix}\right),
\quad
W=\log(I+T)=\sum_{n\geq 1}\frac{(-1)^{n-1}T^n}{n}.
\]
Let $c=ab$ and let
\[
\xi_{1,2}=\frac{c\pm\sqrt{c(c+4)}}{2}
\]
be the solutions of the quadratic equation $\xi^2=c(\xi+1)$.
Since
\[
T^2=ab(T+I)=c(T+I),
\]
by easy induction  and the Vi\`ete formulas
(or using linear recurrence relations arguments)
we obtain
\[
T^n=\frac{1}{\sqrt{c(c+4)}}\left((\xi_1^n-\xi_2^n)T+c(\xi_1^{n-1}-\xi_2^{n-1})I\right).
\]
Hence
\[
W=\frac{1}{\sqrt{c(c+4)}}\sum_{n\geq 1}\frac{(-1)^{n-1}}{n}\left((\xi_1^n-\xi_2^n)T
+c(\xi_1^{n-1}-\xi_2^{n-1})I\right)
\]
\[
=\frac{1}{\sqrt{c(c+4)}}\left((\log(1+\xi_1)-\log(1+\xi_2))T+c\left(\frac{\log(1+\xi_1)}{\xi_1}-\frac{\log(1+\xi_2)}{\xi_2}\right)I\right).
\]
The equation $(1+\xi_1)(1+\xi_2)=1$ implies $\log(1+\xi_1)=-\log(1+\xi_2)$,
\[
W=\frac{\log(1+\xi_1)}{\sqrt{c(c+4)}}\left(2T+c\left(\frac{1}{\xi_1}+\frac{1}{\xi_2}\right)I\right)
=\frac{\log(1+\xi_1)}{\sqrt{c(c+4)}}(2T-cI).
\]
If $W\in R_1$, we would have that $\displaystyle{\frac{\log(1+\xi_1)}{\sqrt{c(c+4)}}}$
is a polynomial in $K[a,b]$ and hence in $K[c]$. In this way
$\log(1+\xi_1)=f(c)\sqrt{c(c+4)}$ for some $f(c)\in K[c]$. Solving the equation $\xi_1^2=c(\xi_1+1)$ with respect to $c$ we obtain
\[
c=\frac{\xi_1^2}{1+\xi_1}=\sum_{n\geq 2}(-1)^n\xi_1^n,\quad \sqrt{c(c+4)}=\frac{\xi_1(2+\xi_1)}{1+\xi_1}.
\]
Since $c$ belongs to the augmentation ideal of the algebra $K[[\xi_1]]$ of formal power series in $\xi_1$
we derive that $K[[c]]\subset K[[\xi_1]]$. Then
\[
\frac{\log(1+\xi_1)}{\sqrt{c(c+4)}}=\frac{1+\xi_1}{\xi_1(2+\xi_1)}\log(1+\xi_1)=f\left(\frac{\xi_1^2}{1+\xi_1}\right),
\]
\[
\log(1+\xi_1)=\frac{\xi_1(2+\xi_1)}{1+\xi_1}f\left(\frac{\xi_1^2}{1+\xi_1}\right).
\]
This is a contradiction because $\log(1+\xi_1)$ is not a rational function in $\xi_1$.
\end{proof}

\begin{example} \label{counterexample to product of inner autos}
Let $L_2$ be the free Lie algebra freely generated by $u,v$  and let $H$ be its factor algebra modulo
the ideal generated by all commutators $[z,u,u],[z,v,v]$, $z\in L_2$.
Then $\exp(\text{\rm ad}u)$ and $\exp(\text{\rm ad}v)$ are well defined inner automorphisms of $H$ but the solution $w$
of the equation $\exp(\text{\rm ad}u)\exp(\text{\rm ad}v)=\exp(\text{\rm ad}w)$ belongs to $\widehat{H}$
and not to $H$.
\end{example}

\begin{proof}
Consider the algebra ${\mathcal M}(L_2)$ of the multiplications of
the Lie algebra $L_2$. It is the associative subalgebra of the algebra of linear operators
of the vector space $L_2$ generated by the operators $\text{ad}z$, $z\in L_2$.
Since the elements of $L_2$ are linear combinations of left normed commutators of $u$ and $v$,
the algebra $\mathcal M(L_2)$ is generated by $\text{ad}u$ and $\text{ad}v$. It is well known that $\mathcal M(L_2)$ is
isomorphic to the two-generated free associative algebra. From the definition of $H$ as the factor Lie algebra
\[
H=L_2/([z,u,u],[z,v,v]\mid z\in L_2)
\]
we derive that the algebra ${\mathcal M}(H)$ of the multiplications of $H$ satisfies the relations
\[
\text{ad}^2u=\text{ad}^2v=0.
\]
The freedom of ${\mathcal M}(L_2)$ implies that the above two equalities are the defining relations of
${\mathcal M}(H)$ and it has the presentation
\[
{\mathcal M}={\mathcal M}(H)=K\langle \text{ad}u,\text{ad}v\mid \text{ad}^2u=\text{ad}^2v=0\rangle.
\]
Let $\widehat{\mathcal M}$ be the completion of $\mathcal M$ with respect to the formal power series topology.
The exponentials
\[
\exp(\text{ad}u)=1+\text{ad}u,\quad \exp(\text{ad}v)=1+\text{ad}v
\]
are inner automorphisms of $H$. Now Example \ref{counterexample to product of exponentials} gives that the
solution $\text{ad}w$, $w\in \widehat{H}$, of the equation $\exp(\text{ad}u)\exp(\text{ad}v)=\exp(\text{ad}w)$
which is in $\widehat{\mathcal M}$ does not belong to $\mathcal M$, i.e., $w$ does not belong to $H$.
\end{proof}

\section*{Acknowledgements}

The second named author is grateful to the Institute of Mathematics and Informatics of
the Bulgarian Academy of Sciences for the creative atmosphere and the warm hospitality during his visit when
most of this project was carried out.

\end{document}